\documentclass[10pt]{amsart}
\usepackage{amsmath, amssymb, amsthm, bm, mathtools, enumitem, caption}
\usepackage{hyperref}

\usepackage[noabbrev,capitalize]{cleveref}
\usepackage{adjustbox}
\crefname{equation}{}{}
\usepackage{graphicx, tikz}
\usepackage[textheight=8.8in, textwidth=6.8in]{geometry}

\usepackage{microtype}

\newtheorem{theorem}{Theorem}[section]
\newtheorem{lemma}[theorem]{Lemma}
\newtheorem{corollary}[theorem]{Corollary}
\newtheorem{proposition}[theorem]{Proposition}

\newtheorem*{conjecture*}{Conjecture}

\theoremstyle{definition}

\theoremstyle{remark}
\newtheorem*{remark}{Remark}

\newtheorem*{example}{Example}

\numberwithin{equation}{section}

\newcommand{\R}{\mathbb R}
\newcommand{\N}{\mathbb N}

\DeclareMathOperator{\Tr}{Tr}

\newcommand{\leg}[2]{\genfrac{(}{)}{}{}{#1}{#2}}

\DeclareMathOperator{\sign}{sign}

\newcommand{\HH}{\mathbb H}

\newcommand{\C}{\mathbb C}
\newcommand{\SL}{\mathrm{SL}}

\newcommand{\RR}{\mathbb R}
\newcommand{\Z}{\mathbb Z}

\title[Pentagonal number recurrence relations for $p(n)$]{Pentagonal number recurrence relations for $p(n)$}

\author{Kevin Gomez, Ken Ono, Hasan Saad \and Ajit Singh}
\address{Dept. of Mathematics, University of Virginia, Charlottesville, VA 22904}
\email{vhe4ht@virginia.edu}
\email{ken.ono691@virginia.edu}

\address{Dept. of Mathematics, Louisiana State University, Baton Rouge, LA 70803}
\email{hsaad@lsu.edu}

\address{Dept. of Mathematics, University of Virginia, Charlottesville, VA 22904}
\email{ajit18@iitg.ac.in}

\begin{document}
\begin{abstract} We revisit Euler's partition function recurrence, which asserts, for integers $n\geq 1,$ that
$$
p(n)=p(n-1)+p(n-2)-p(n-5)-p(n-7)+\dots = \sum_{k\in \Z\setminus \{0\}} (-1)^{k+1} p(n-\omega(k)),
$$
where $\omega(m):=(3m^2+m)/2$ is the $m$th pentagonal number. We prove that this classical result is
the $\nu=0$ case of an infinite family of ``pentagonal number'' recurrences. For each $\nu\geq 0,$ we prove for positive $n$ that
 $$
p(n)=\frac{1}{g_{\nu}(n,0)}\left(\alpha_{\nu}\cdot \sigma_{2\nu-1}(n)+ \Tr_{2\nu}(n) +\sum_{k\in \Z\setminus \{0\}} (-1)^{k+1} g_{\nu}(n,k)\cdot p(n-\omega(k))\right),
$$
where $\sigma_{2\nu-1}(n)$ is a divisor function, $\Tr_{2\nu}(n)$ is the $n$th weight $2\nu$ Hecke trace of values of special twisted quadratic Dirichlet series, and each $g_{\nu}(n,k)$ is a polynomial in $n$ and $k.$ 
The $\nu=6$ case can be viewed as a partition theoretic formula for Ramanujan's tau-function, as we have
$$
\Tr_{12}(n)=-\frac{33108590592}{691}\cdot \tau(n).
$$
 \end{abstract}

\maketitle

\section{Introduction and Statement of Results}
A {\it partition} of $n$ is any nonincreasing sequence of positive integers that sums to $n.$  The number of partitions of $n$ is denoted $p(n)$ (by convention, we agree that $p(0):=1$). To compute these numbers, Euler famously offered the recurrence
(see p. 12 of \cite{Andrews})
\begin{equation}\label{EulerRecurrence}
p(n)=p(n-1)+p(n-2)-p(n-5)-p(n-7)+ \dots = \sum_{k\in \Z \setminus \{0\}} (-1)^{k+1}p(n-\omega(k)),
\end{equation}
where $\omega(m):=(3m^2+m)/2$ is the $m$th {\it pentagonal number}. This recurrence, which is one of the most efficient methods for computing partition numbers, is obtained by multiplying the generating function
$$
\sum_{n=0}^{\infty} p(n)q^n = \prod_{n=1}^{\infty}\frac{1}{1-q^n}=1+q+2q^2+3q^3+5q^4+7q^5+\dots
$$
with Euler's {\it Pentagonal Number Theorem}  (see p. 11 of \cite{Andrews}) infinite product 
$$
\prod_{n=1}^{\infty} (1-q^n)=\sum_{k\in \Z} (-1)^k q^{\frac{3k^2+k}{2}}=1-q-q^2+q^5+q^7-q^{12}+\dots.
$$
Indeed, (\ref{EulerRecurrence}) follows by comparing the coefficient of $q^n$ on both sides of the trivial identity
\begin{equation}\label{Trivial}
\left(\sum_{n=0}^{\infty}p(n)q^n\right)\cdot \left(
\sum_{k\in \Z}(-1)^k q^{\frac{3k^2+k}{2}}\right)=1.
\end{equation}

We prove that Euler's recurrence is the $\nu=0$ case of an infinite family of rich recurrence relations satisfied by the partition numbers. To make this precise, we make use of {\it Dedekind's eta-function}
\begin{equation}\label{DedekindDefinition}
\eta(\tau):=q^{\frac{1}{24}}\prod_{n=1}^{\infty}(1-q^n)=\sum_{k\in \Z} (-1)^k q^{\frac{(6k+1)^2}{24}},
\end{equation}
where $\tau\in \HH,$ the upper half of the complex plane, and where $q:=e^{2\pi i \tau}$.
To define these relations, we employ the {\it falling factorials}
\begin{equation}
(x)_m:=\begin{cases}  x(x-1)\cdots (x-m+1) \ \ \ \ &{\text {\rm if $m\geq 1$}},\\
1 \ \ \ \ \ &{\text {\rm if $m=0$}},\\
                                   \frac{1}{(x)_{-m}} \ \ \ \ \ &{\text {\rm if $m\leq -1$}},
                                 \end{cases}
\end{equation}
and  the differential operator $D :=\frac{1}{2\pi i}\frac{d}{d\tau} = q\frac{d}{dq}.$
 For non-negative integers $\nu,$ we define 
\begin{equation}\label{Pnu}
    P_{\nu}(\tau) := \frac{(2\nu-1)(2\nu-2)_{\nu-1}^2}{2^{2\nu-2}} \sum_{\substack{r,s \geq 0 \\ r + s = \nu}} (-1)^r \frac{(2s-1)}{(2r)! (2s)!}\cdot D^r\left(\frac{1}{\eta(\tau)}\right) \cdot D^s(\eta(\tau)).
\end{equation}
These $q$-series arise in the theory of modular forms as the {\it Rankin--Cohen} brackets 
$P_{\nu}(\tau):=[1/\eta(\tau),\eta(\tau)]_{\nu}$ (for example, see \cite{Cohen, Rankin, Zagier}). The goal of this paper is to explicitly determine all of these $q$-series using the general theory of Rankin--Cohen brackets (see Section~\ref{RCSection}), using both standard and new objects in the theory of modular forms.

The first four  $P_{\nu}(\tau)$ are
\begin{displaymath}
\begin{split}
P_0(\tau)&=\frac{1}{\eta(\tau)}\cdot\eta(\tau)=1,\\
P_1(\tau)&= -\frac{1}{2}\left(\frac{1}{\eta(\tau)} D(\eta(\tau))+\eta(\tau)D\left(\frac{1}{\eta(\tau)}\right)\right),\\
P_2(\tau)&=-\frac{1}{8}\left(\frac{1}{\eta(\tau)} D^2(\eta(\tau))+6 D\left(\frac{1}{\eta(\tau)}\right)D\left(\eta(\tau)\right)-3\eta(\tau)D^2\left(\frac{1}{\eta(\tau)}\right)\right),\\
P_3(\tau)&=-\frac{1}{16}\left(\frac{1}{\eta(\tau)}D^3(\eta(\tau))+15D\left(\frac{1}{\eta(\tau)}\right)D^2(\eta(\tau)) - 45 D^2\left(\frac{1}{\eta(\tau)}\right)D(\eta(\tau)) +5\eta(\tau)D^3\left(\frac{1}{\eta(\tau)}\right)\right).\\
\end{split}
\end{displaymath}
Furthermore, a straightforward calculation shows that
\begin{align}\label{series_side}
	P_{\nu}(\tau)= \sum_{\substack{n\geq 0\\ k\in \Z}} (-1)^{k+1} g_{\nu}(n,k)\cdot p(n-\omega(k))q^n,
\end{align}
where
\begin{equation}\label{RCpolynomials}
    g_\nu(n,k) :=\frac{(2\nu-1)(2\nu-2)_{\nu-1}^2}{2^{2\nu-2}}\sum_{r=0}^{\nu} (-1)^{\nu+r}\frac{ (2\nu-2r-1)}{ (2r)! (2\nu-2r)!}\cdot (6k+1)^{2r}(24n - (6k+1)^2)^{\nu - r}.
\end{equation} 

In analogy with the derivation of (\ref{EulerRecurrence}) from (\ref{Trivial}), which is equivalent to the fact that $P_0(\tau)=1,$ we obtain a partition function recurrence relation
 for every  $\nu\geq 1$ because we are able to explicitly determine $P_{\nu}(\tau)$.
With  this goal in mind, we  begin with the following crucial fact  that stems from the general theory of Rankin--Cohen brackets.

\begin{theorem}\label{theorem1} If $\nu\geq 0$, then
$P_{\nu}(\tau)$ is a weight $2\nu$ holomorphic modular form on $\SL_2(\Z).$
\end{theorem}

\begin{remark}
We have that $P_0(\tau)=1$ and $P_1(\tau)=0,$ as there are no nontrivial holomorphic modular forms of weight 0 and 2. By
comparing coefficients, these identities give two proofs of (\ref{EulerRecurrence}).  
\end{remark}

For $\nu \in \{2, 3, 4, 5, 7\}$, we find that $P_{\nu}(\tau)$ is a scalar multiple of $E_{2\nu}(\tau),$ where 
\begin{equation}
E_{2k}(\tau):=1-\frac{4k}{B_{2k}}\sum_{n=1}^{\infty} \sigma_{2k-1}(n)q^n
\end{equation}
 is the weight $2k$ Eisenstein series.
As a consequence, we obtain the following simple recurrence relations.

\begin{corollary}\label{NoCuspForms} If $\nu \in \{2, 3, 4, 5, 7\}$ and $n$ is a positive integer, then
$$
        p(n) =\frac{1}{g_{\nu}(n,0)}\left( -\frac{4\nu}{B_{2\nu}}    \binom{2\nu - 2}{\nu - 2}\sigma_{2\nu - 1}(n)  + \sum_{k \in \Z\setminus\{0\}} (-1)^{k+1} g_\nu(n,k)\,p(n - \omega(k))\right),
$$        
 where $B_n$ is the $n^{\text{th}}$ Bernoulli number and $\sigma_{2\nu-1}(n):=\sum_{d\mid n}d^{2\nu-1}.$
\end{corollary}

\begin{remark} For $\nu\geq 0$, we have that $g_{\nu}(n,0)$ is nonzero for positive  $n.$ To see this, observe that 
\begin{align*}
    	g_{\nu}(n,0) = \frac{(2\nu-1)(2\nu-2)_{\nu-1}^2}{2^{2\nu-2}(2\nu)!}\cdot  h_{\nu}(24n - 1)
\end{align*}
for $h_{\nu}(n)$ a polynomial with integer coefficients and constant term $\pm 1$. Since $24n - 1 > 1$ whenever $n \geq 1$,  there is a prime $\ell\mid (24n-1)$, and so we have that $h_{\nu}(24n - 1)\equiv \pm 1\pmod \ell$.  The desired nonvanishing is an immediate consequence.
\end{remark} 

\begin{example} For positive integers $n$, Corollary~\ref{NoCuspForms} with $\nu=2$ gives the relation
\begin{displaymath}
\begin{split}
p(n)&= \frac{240\sigma_3(n)}{216n^2-36n+1}\\
&\ \ \ \ \ \  +\frac{1}{216n^2-36n+1}\sum_{k\in \Z \setminus \{0\}}
(-1)^{k+1} \left(216n^2-36(6k+1)^2n+(6k+1)^4\right) p(n-\omega(k)).
\end{split}
\end{displaymath}
\end{example}

For $\nu \in \{6, 8, 9, 10, 11, 13\},$ the space of weight $2\nu$ cusps forms is $S_{2\nu}=\C \Delta_{2\nu}(\tau),$ where
\begin{equation}\label{Dim1CuspForm}
\Delta_{2\nu}(\tau):=q+\sum_{n=2}^{\infty} \tau_{2\nu}(n)q^n:= \begin{cases}
\Delta(\tau) \ \ \ \ \ &{\text  {\rm if $\nu=6,$}}\\
\Delta(\tau)E_4(\tau) \ \ \ \ &{\text {\rm  if $\nu=8$}},\\
\Delta(\tau)E_6(\tau) \ \ \ \ &{\text {\rm  if $\nu=9$,}}\\
\Delta(\tau)E_4(\tau)^2 \ \ \ \ \ &{\text {\rm  if $\nu=10$,}}\\
\Delta(\tau)E_4(\tau)E_6(\tau) \ \ \ \ \ &{\text {\rm  if $\nu=11$,}}\\
\Delta(\tau)E_4(\tau)^2 E_6(\tau)\ \ \ \ \ &{\text {\rm  if $\nu=13,$}}
\end{cases}
\end{equation}
with  $\Delta(\tau):=\eta(\tau)^{24}.$ For these $\nu,$ we obtain the following recurrence relations.

\begin{corollary}\label{Dim1}  If $\nu \in \{6, 8, 9, 10, 11, 13\}$ and $n$ is a positive integer, then
\begin{displaymath}
\begin{split}
        p(n) =\frac{1}{g_{\nu}(n,0)}\left( -\frac{4\nu}{B_{2\nu}}    \binom{2\nu - 2}{\nu - 2}\sigma_{2\nu - 1}(n) +\beta_{\nu}\cdot\tau_{2\nu}(n) + \sum_{k \in \Z\setminus\{0\}} (-1)^{k+1} g_\nu(n,k)\,p(n - \omega(k))\right),
\end{split}
\end{displaymath}
where we have
\begin{displaymath}
		\beta_{\nu}:= \begin{cases} -\frac{33108590592}{691} \ \ \ \ \ &{\text {if $\nu=6,$}}\\ \ \\
			-\frac{187167592415232}{3617} \ \ \ \ \ &{\text  {if $\nu=8,$}}\\ \ \\
			-\frac{28682634201661440}{43867} \ \ \ \ \ &{\text {if $\nu=9,$}}\\ \ \\
			-\frac{8294726176465158144}{174611} \ \ \ \ \ &{\text {if $\nu=10,$}}\\ \ \\
			-\frac{101475065073734516736}{77683} \ \ \ \ \ &{\text {if $\nu=11,$}}\\ \ \\
			-\frac{1195065734266339700244480}{657931} \ \ \ \ \ &{\text {if $\nu=13.$}}
		\end{cases}
	\end{displaymath} 
\end{corollary}

\begin{example}
Corollary~\ref{Dim1} with $\nu=6$ gives a partition number formula for Ramanujan's $\tau$-function $\tau(n)=\tau_{12}(n)$. Namely, if $n$ is a positive integer, then we have
\begin{displaymath}
\begin{split}
	&\tau(n)= \\
	& \ \ \ \ \frac{2275}{5474304} \sigma_{11}(n)-\frac{691}{33108590592}g_6(n,0)p(n) 
	 +\frac{691}{33108590592}\sum_{k \in \Z \setminus \{0\}} (-1)^{k+1} g_6(n,k)\,p{\left(n - \omega(k)\right)}.
\end{split}
\end{displaymath}
It is pleasing to note that this formula immediately implies Ramanujan's celebrated congruence
$$
\tau(n)\equiv \sigma_{11}(n)\pmod{691}.
$$
\end{example}

Finally, we turn to the case of general $\nu,$ where $S_{2\nu}$ is non-empty, extending Corollary~\ref{Dim1}.  These relations also are rich in structure, and involve new infinite weighted sums of twisted Dirichlet series which are easily described in terms of the normalized Hecke eigenforms and their Petersson norms. To make this precise, suppose that
$$f(\tau):=q+\sum\limits_{n\geq 2}a_f(n)q^n\in S_{2\nu}$$ is a Hecke eigenform.  For $s\in\C$ with $\Re(s)\geq2\nu+1,$ we define the {\it twisted quadratic Dirichlet series}
\begin{equation}
D(f;s):=\sum\limits_{n\geq 1}\frac{\leg{12}{n}\cdot a_f\left(\frac{n^2-1}{24}\right)}{n^s},
\end{equation}
where $\leg{12}{\cdot}$ is the Kronecker symbol. Furthermore, if $\nu\geq 2, 0\leq j\leq \nu-2$ and $m\geq 0,$  we define constants
\begin{equation}
	\beta(\nu,j,m):=\frac{(-1)^{j+1}\Gamma(\nu-\frac{1}{2})\Gamma(\nu+\frac{1}{2})}{2\sqrt{\pi}\cdot \Gamma\left(\frac{5}{2}\right)}\cdot\left(\frac{6}{\pi}\right)^{2\nu-1}\frac{(2\nu+m-2)!}{j!\cdot m!\cdot (2\nu-j-2)!}\cdot\frac{(\nu-j-1)^{\overline{\nu}}\left(\frac{3}{2}\right)^{\overline{j}}}{(-\frac{1}{2}-j)^{\overline{\nu}}\cdot(\frac{5}{2})^{\overline{j}}},
\end{equation}
where $\Gamma(\cdot)$ is the usual Gamma-function, and where we employ the {\it rising factorial} notation
\begin{equation}
(x)^{\overline{j}}:=\begin{cases} x(x+1)\cdots (x+j-1)  \ \ \ \ \ &{\text {\rm if}}\ j\geq1,\\ 
1  \ \ \ \ \ &{\text {\rm if } j=0}.
\end{cases}
\end{equation}
Finally, we recall the Petersson norm (see Section~\ref{PeterssonSubsection})
$$
||f||:=\int\int_{\HH/\SL_2(\Z)} |f|^2 y^{2\nu-2}dxdy.
$$

We obtain an infinite family (see Theorem~\ref{theorem2}) of pentagonal number partition function recurrence relations in terms of divisor functions and ``Hecke traces'' of an infinite sum of values of twisted quadratic Dirichlet series. This general result gives Corollaries 1.2 and 1.3 as special cases. To make this precise, if $f\in S_{2k},$ the space of weight $2k$ cusp forms on $\SL_2(\Z),$ then we define these sums of values of Dirichlet series by
\begin{equation}\label{InfiniteSum}
D_f:= \sum\limits_{j=0}^{k-2}\sum\limits_{m=0}^\infty \beta(k ,j,m)\cdot D(f;2k+1+2m+2j),
\end{equation}
and we define the {\it weight $2k$ Hecke trace} by
\begin{equation}
	\Tr_{2k}(n):=\sum\limits_{f}  a_f(n)\cdot \frac{D_f}{||f||},
\end{equation}
where the sum runs over the normalized Hecke eigenforms $f\in S_{2k}.$

\begin{theorem}\label{theorem2} If $\nu \geq 6,$ with $\nu \neq 7,$ then for positive integers $n$ we have
	$$
	p(n)=\frac{1}{g_{\nu}(n,0)}\left(-\frac{4\nu}{B_{2\nu}}\binom{2\nu-2}{\nu-2}\sigma_{2\nu-1}(n)+ \Tr_{2\nu}(n) +\sum_{k\in \Z\setminus \{0\}} (-1)^{k+1} g_{\nu}(n,k)\cdot p(n-\omega(k))\right).
	$$
\end{theorem}

\begin{example}  Here we consider the $\nu=6$ case of Theorem~\ref{theorem2}, where we have
$$
	p(n)= \frac{1}{g_6(n,0)}\left(\frac{13759200}{691}\cdot \sigma_{11}(n)+ \Tr_{12}(n)+\sum_{k\in \Z\setminus \{0\}} (-1)^{k+1} g_{6}(n,k)\cdot p(n-\omega(k))\right).
$$
Thanks to Corollary~\ref{Dim1}, we find that
$$
\Tr_{12}(n)=-\frac{33108590592}{691}\cdot \tau(n).
$$
Therefore, using the fact that $||\Delta(\tau)||\approx  1.035362\cdot 10^{-6},$ 
we have
$$
{D}_{\Delta}:= \sum\limits_{j=0}^{4}\sum\limits_{m=0}^\infty \beta(6,j,m)\cdot D(\Delta;13+2m+2j) 
 = -\frac{33108590592}{691} \cdot ||\Delta(\tau)||=-49.608362\dots.
$$
On the other hand, this value is easily confirmed numerically using the definition
$$
D(\Delta,N;s):=\sum\limits_{n=1}^{N}\frac{\leg{12}{n}\cdot a_f\left(\frac{n^2-1}{24}\right)}{n^s},
$$
and
$$
\widehat{D}_{\Delta}(M,N):=\sum\limits_{j=0}^{4}\sum\limits_{m=0}^{M} \beta(6,j,m)\cdot D(\Delta, N;13+2m+2j).
$$
A short computer computation gives
$\widehat{D}_{\Delta}(100, 2000)=-49.608382\dots.$

\end{example}

\begin{example} Theorem~\ref{theorem2} can be employed to exactly compute normalized weighted sums of twisted Dirichlet series values using only the knowledge of a handful of partition numbers.
For example, we consider the  $\nu=12$ case. Theorem~\ref{theorem2} gives
\begin{equation}\label{expression}
p(n)=\frac{1}{g_{12}(n,0)}\left(\frac{84736491840}{236364091}\sigma_{23}(n)+ \Tr_{24}(n) +\sum_{k\in \Z\setminus \{0\}} (-1)^{k+1} g_{12}(n,k)\cdot p(n-\omega(k))\right).
\end{equation}
The weight 24 Hecke trace $\Tr_{24}(n)$ is computed using the two normalized Hecke eigenforms
\begin{displaymath}
\begin{split}
  f_1=\sum_{n=1}^{\infty}a_{f_1}(n)q^n&:=\Delta(\tau)E_4(\tau)^3+(-156-12\sqrt{144169})\Delta(\tau)^2=q+(540-12\sqrt{144169})q^2+\dots,\\
  f_2=\sum_{n=1}^{\infty}a_{f_2}(n)q^n&:=\Delta(\tau)E_4(\tau)^3+(-156+12\sqrt{144169})\Delta(\tau)^2=q+(540+12\sqrt{144169})q^2+\dots.
\end{split}
\end{displaymath}
Therefore, we have that
$$
\Tr_{24}(n):=  a_{f_1}(n)\cdot \frac{D_{f_1}}{||f_1||} +  a_{f_2}(n)\cdot \frac{D_{f_2}}{||f_2||},
$$
which in turn means that
 $T_{24}(\tau):=\sum_{n=1}^{\infty}\Tr_{24}(n)q^n\in S_{24}.$ Using
  the $n=1$ and $n=2$ cases of (\ref{expression}), with $p(1)=1$ and $p(2)=2,$ we obtain
$$
T_{24}(\tau)= -\frac{11762326506193377107116032}{236364091}q- \frac{22599437869751987230702829568}{236364091}q^2+\dots.
$$
Using the first two coefficients $f_1$ and $f_2$, we obtain two linear equations, whose solution are exactly
$$
\frac{D_{f_1}}{||f_1||} = {\color{black}\frac{-5881163253096688553558016}{236364091} + \frac{{\color{black} 676990898183648483035840512}}{236364091\sqrt{144169}}}$$
and
$$
\frac{D_{f_2}}{||f_2||}= {\color{black}\frac{-5881163253096688553558016}{236364091} - \frac{{\color{black} 676990898183648483035840512}}{236364091\sqrt{144169}}}.
$$
\end{example}

To obtain the results in this paper,
we make use of the theory of Rankin--Cohen bracket operators, which in turn implies that each $P_{\nu}(\tau)$ is a weight $2\nu$ holomorphic modular form (i.e. Theorem~\ref{theorem1}).  Corollaries~\ref{NoCuspForms} and ~\ref{Dim1} follow from straightforward calculations as $\dim_{\C}(S_{2\nu})\in \{0, 1\}.$
The proof of Theorem~\ref{theorem2} is far more involved, and it requires the explicit calculation of the Petersson inner product of $P_{\nu}(\tau)$ with the Eisenstein series $E_{2\nu}(\tau)$  and with each Hecke eigenform in $S_{2\nu}$. To carry out these calculations, we use the method of ``unfolding'', which applies because we are able to describe the generating  function for $p(n)$ as a Poincar\'e series. With such inner products, we are then able to decompose $P_{\nu}(\tau)$ as a sum of a multiple of $E_{2\nu}(\tau)$ with Hecke traces of an infinite sum of values of  twisted quadratic Dirichlet series.

This paper is organized as follows. We prove Theorem~\ref{theorem1} and Corollaries~\ref{NoCuspForms} and~\ref{Dim1} in Section~2. Although Theorem~\ref{theorem1} is not new, for completeness, we offer a simple and short proof based on an argument of Zagier, which relies only on classical identities of Ramanujan. In Section~3, which includes most of the work in this paper, we recall and derive new facts about relevant Poincar\'e series, the method of unfolding, Petersson inner products, and special twisted quadratic Dirichlet series. Finally, in Section~4 we use these results to prove Theorem~\ref{theorem2}.

\section*{Acknowledgements}
\noindent The authors thank the referees for their helpful comments on an earlier version of this paper.
 The second author is grateful for the support of the Thomas Jefferson Fund and the NSF
(DMS-2002265 and DMS-2055118). The fourth author is grateful for the support of a Fulbright Nehru Postdoctoral Fellowship.

\section{Proof of Theorem~\ref{theorem1} and Corollaries~\ref{NoCuspForms} and \ref{Dim1}}

Here we give an elementary proof of Theorem~\ref{theorem1}, following an argument of Zagier, which relies on classical identities of Ramanujan. Then we apply these results to obtain Corollaries~\ref{NoCuspForms} and \ref{Dim1}.

\subsection{Proof of Theorem~\ref{theorem1}}

To prove Theorem \ref{theorem1}, we need three lemmas. The first lemma recalls some of Ramanujan's famous differential formulas involving $E_2(\tau), E_4(\tau)$ and $E_6(\tau)$.
\begin{lemma}\cite[p. 181]{Ramanujan}\label{lem1}
 We have that
\begin{align*}
&D(E_2(\tau))=\frac{E_2(\tau)^2-E_4(\tau)}{12},\\
& D(E_4(\tau))=\frac{E_2(\tau)E_4(\tau)-E_6(\tau)}{3},\\
&D(E_6(\tau))=\frac{E_2(\tau)E_6(\tau)-E_4(\tau)^2}{3}.
\end{align*}
\end{lemma}

We will also require the modular transformation laws for $E_2(\tau), \eta(\tau),$ and $1/\eta(\tau)$.
\begin{lemma}\label{lem2}
If $\gamma=\begin{bmatrix}
			a  &  b \\
			c  &  d      
		\end{bmatrix}\in \mathrm{SL}_2(\Z),$ then we have 
\begin{align*}
&E_2\left(\gamma(\tau)\right)=(c\tau+d)^2E_2(\tau)+\frac{6c}{\pi i}(c\tau+d), \\
&\eta\left(\gamma(\tau)\right)=\varepsilon(\gamma)(cz+d)^{\frac{1}{2}}\eta(\tau),\\ 
&\frac{1}{\eta\left(\gamma(\tau)\right)}=\overline{\varepsilon}(\gamma)(cz+d)^{\frac{-1}{2}}\frac{1}{\eta(\tau)},
\end{align*}
where the multiplier system $\varepsilon(\cdot)$ is given by
\begin{displaymath}	
	\varepsilon(\gamma):= \begin{cases}  \left(\frac{d}{|c|}\right)\exp\left(\frac{\pi i}{12}(c(a+d-3)-bd(c^2-1))\right)\ \ \ \ \ &{\text {if $c$ is odd,}}\\
	\left(\frac{c}{|d|}\right)\exp\left(\frac{\pi i}{12}(c(a-2d)-bd(c^2-1)+3d-3)\right)\varepsilon(c,d)\ \ \ \ \ \ &{\text  {if $c$ is even}}.
	\end{cases}
\end{displaymath}
Here $\left(\frac{c}{d}\right)$ is the Kronecker--Legendre symbol and 
$$\varepsilon(c,d):= \begin{cases} -1 \ \ \ \ \ &{\text {when $c\leq 0$ and $d<0$}}\\
\  \ 1 \ \ \ \ \ &{\text  {otherwise}}.
 \end{cases}
 $$
\end{lemma}
Finally, as the series $P_\nu(\tau)$ is assembled by differentiating $\eta(z)$ and $1/\eta(\tau)$, we shall require the following useful formulas.

	\begin{lemma} \label{ZagierLemma} If $\nu$ is a nonnegative integer, then the following are true.
	
\noindent
	(1)	We have that
		$$
			\frac{D^\nu(\eta(\gamma(\tau)))}{\nu!\Gamma(\nu+1/2)} =\varepsilon(\gamma) \sum_{m=0}^{\nu} \frac{(2\pi i c)^{\nu - m}(c\tau + d)^{\nu+m+1/2}}{(\nu-m)!} \frac{D^m(\eta(\tau))}{m!\Gamma(m+1/2)}.
$$

\noindent
(2) We have that
$$
			\frac{D^\nu\left(\frac{1}{\eta(\gamma(\tau))}\right)}{\nu!\Gamma(\nu-1/2)} =\overline{\varepsilon}(\gamma) \sum_{m=0}^{\nu} \frac{(2\pi i c)^{\nu - m}(c\tau + d)^{\nu+m-1/2}}{(\nu-m)!} \frac{D^m\left(\frac{1}{\eta(\tau)}\right)}{m!\Gamma(m-1/2)}.
$$

	\end{lemma}
	
	\begin{proof}
		We prove (1)  by induction on $\nu$, and we note that the proof of (2) follows similarly.  The cases where $\nu \in \{ 0, 1\}$ follow by Lemma \ref{lem2}. For the induction step, let 
		\begin{align*}
			\frac{D^\nu(\eta(\gamma(\tau)))}{\nu!\Gamma(\nu+1/2)} &= \varepsilon(\gamma)\sum_{m=0}^{\nu} \frac{(2\pi i c)^{\nu - m}(c\tau + d)^{\nu+m+1/2}}{(\nu-m)!} \frac{D^m(\eta(\tau))}{m!\Gamma(m+1/2)}.
		\end{align*}
		Differentiating above both sides with respect to $\tau$ to get  
		\begin{align*}
			\frac{D^{\nu+1}(\eta(\gamma(\tau)))}{\nu!\Gamma(\nu+1/2)} &= \varepsilon(\gamma)\sum_{m=0}^{\nu} \frac{(2\pi i c)^{\nu - m+1}(\nu+m+1/2)(c\tau + d)^{\nu+m+3/2}}{(\nu-m)!} \frac{D^m(\eta(\tau))}{m!\Gamma(m+1/2)}\\
			&+\varepsilon(\gamma)\sum_{m=0}^{\nu} \frac{(2\pi i c)^{\nu - m}(c\tau + d)^{\nu+m+5/2}}{(n-m)!} \frac{D^{m+1}(\eta(\tau))}{m!\Gamma(m+1/2)}.
		\end{align*}
		After straightforward algebraic manipulation, we obtain
		\begin{align*}
			\frac{D^{\nu+1}(\eta(\gamma(\tau)))}{\nu!\Gamma(\nu+1/2)} &=  \varepsilon(\gamma)\frac{(2\pi i c)^{\nu+1}(\nu+1/2)(c\tau + d)^{\nu+3/2}}{\nu!\Gamma(1/2)} f(\tau)+\varepsilon(\gamma)\frac{(c\tau + d)^{2\nu+5/2}}{\nu!\Gamma(\nu+1/2)} D^{\nu+1}(\eta(\tau))\\
			&+\varepsilon(\gamma)(\nu+1)(\nu+1/2)\sum_{i=1}^{\nu}\frac{(2\pi i c)^{\nu - i+1}(c\tau + d)^{\nu+i+3/2}}{(\nu-i+1)!} \frac{D^i(\eta(\tau))}{i!\Gamma(i+1/2)}.
		\end{align*}
		The claim follows by multiplying both sides by $\frac{1}{(\nu+1)(\nu+1/2)}$.  
	\end{proof}

	\begin{proof}[Proof of Theorem~\ref{theorem1}] 
		We consider the formal power series (for example, see page 58 of Zagier \cite{Zagier})
		\begin{align*}
			\tilde{f}(\tau,X) &:= \sum_{\nu=0}^{\infty} \frac{D^\nu(\eta(\tau))}{\nu!\Gamma(\nu+1/2)}(2\pi i X)^\nu, \\
			\tilde{g}(\tau,X) &:= \sum_{n=0}^{\infty} \frac{D^\nu\left(\frac{1}{\eta(\tau)}\right)}{\nu!\Gamma(\nu-1/2)}(2\pi i X)^\nu.
		\end{align*}
		A straightforward calculation reveals that
		\begin{align*}
			\tilde{f}(\tau,-X)\tilde{g}(\tau,X) = \sum_{\nu=0}^{\infty} \frac{P_\nu(\tau)}{\Gamma(\nu+1/2)\Gamma(\nu-1/2)} (2\pi i X)^\nu.
		\end{align*}
		Applying Lemma \ref{ZagierLemma} to the $\nu$th coefficients of $\tilde{f}$ and $\tilde{g}$, then using the Taylor expansion for $e^X$, we obtain the following transformation laws for $\gamma \in \mathrm{SL}_2(\Z)$:
		\begin{align*}
			\tilde{f}\left(\gamma(\tau), \frac{X}{(c\tau + d)^2}\right) &= \varepsilon(\gamma)(c\tau + d)^{1/2} e^{-cX/(c\tau + d)} \tilde{f}(\tau,X), \\
			\tilde{g}\left(\gamma(\tau),\frac{X}{(c\tau + d)^2}\right) &= \overline{\varepsilon}(\gamma)(c\tau + d)^{-1/2} e^{cX/(c\tau + d)} \tilde{g}(\tau,X).
		\end{align*}
		This in turn shows that $P_\nu(\tau)$ is a holomorphic modular form of weight $2\nu$ for all $\nu \geq 0$.  In particular, thanks to Lemma~\ref{lem1} and the identities
$$
D(\eta(\tau))=\frac{1}{24}E_2(\tau)\eta(\tau)~~\text{and}~~D\left(\frac{1}{\eta(\tau)}\right)=-\frac{1}{24}E_2(\tau)\frac{1}{\eta(\tau)},
$$		
it follows that $P_\nu(\tau)$ is a polynomial in $E_2(\tau)$, $E_4(\tau)$, and $E_6(\tau),$ which in turn must be a polynomial in $E_4(\tau)$ and $E_6(\tau)$ as they generate all holomorphic modular forms on $\SL_2(\Z)$ 
(for example, see Theorem 1.23 of \cite{CBMS}).
\end{proof}

\subsection{Proof of Corollaries~\ref{NoCuspForms} and \ref{Dim1}}
For $\nu\in \{2, 3, 4, 5, 7\},$ we have that $S_{2\nu}=\{0\}.$ Therefore, it follows that $P_{\nu}(\tau)$ is a constant multiple of the Eisenstein series $E_{2\nu}(\tau)$.  For Corollary~\ref{Dim1}, we have that $S_{2\nu}=\C \Delta_{2\nu}(\tau)$ (see (\ref{Dim1CuspForm}), and so $P_{\nu}(\tau)$ is a linear combination of $E_{2\nu}(\tau)$ and $\Delta_{2\nu}(\tau)$. By explicitly computing these expressions, these two corollaries follow from (\ref{series_side}) by comparing coefficients of $q^n.$

\section{Petersson inner products and twisted  quadratic Dirichlet series}

Theorem~\ref{theorem1} asserts that each $P_{\nu}(\tau)$ is a weight  $2\nu$ holomorphic modular form on $\SL_2(\Z)$. As it is straightforward to compute the constant terms of these forms, which in turn determines the Eisenstein series pieces,  the main objective underlying the proof of Theorem~\ref{theorem2} is the explicit determination of the cuspidal pieces. This task is rather involved. The first step is the determination of a new representation of $1/\eta(\tau),$ the generating function for $p(n),$ which we offer as a Poincar\'e series. We carry out all of the subsequent calculations related to $P_{\nu}(\tau)$ with this knowledge. The remaining steps lead up to and culminate with the method of ``unfolding,'' which uses the theory of Petersson inner products to express these cuspidal components in terms of Hecke eigenvalues and the desired infinite sums of values of twisted  quadratic Dirichlet series. In this section we offer these steps in detail.

\subsection{Harmonic Maass form Poincar\'e series} 

The main observation that allows us to prove Theorem~\ref{theorem2} is the realization that the weight $-1/2$ modular form $1/\eta(\tau)$ has a convenient description as a Poincar\'e series. Here we recall the construction of negative weight harmonic Maass form Poincar\'e series as offered in the work of Bringmann and the second author \cite{BringmannOnoPoincare}. We then specialize this work to the case of the weakly holomorphic modular form $1/\eta(\tau)$.

We first recall notation from the theory of modular and harmonic Maass forms (see Chapter 4 of \cite{HMFBook}), which we employ to carry out this construction. Let $\leg{c}{d}$ be the extended Legendre symbol and let $\sqrt{\cdot}$ be the principal branch of the holomorphic square root, and for odd integers $d$ set
$$
\varepsilon_d := \begin{cases}
	1 &\ \ \ \text{if }d\equiv1\pmod{4}, \\
	i &\ \ \ \text{if }d\equiv3\pmod{4}.
\end{cases}
$$
Furthermore, let $\varepsilon(\cdot)$ be a multiplier system on $\SL_2(\Z).$ If $\HH$ is the upper half of the complex plane, $k\in\frac{1}{2}\Z,$  and $f:\HH\to\C$ is a smooth function, then for all $\gamma=\begin{pmatrix}
	a & b \\
	c & d
\end{pmatrix}\in\SL_2(\Z),$ we define
$$
(f|_{k,\varepsilon}\gamma)(\tau):=\begin{cases}
	(c\tau+d)^{-k}\varepsilon(\gamma)^{-1}f\left(\frac{a\tau+b}{c\tau+d}\right) &\ \ \ \text{if }k\in\Z, \\
	\leg{c}{d}\varepsilon_d^{2k}(c\tau+d)^{-k}\varepsilon(\gamma)^{-1}f\left(\frac{a\tau+b}{c\tau+d}\right)&\ \ \ \text{if }k\in\frac{1}{2}+\Z.
\end{cases}
$$

We consider negative weight $k$ harmonic Maass  Poincar\'e series, which  are obtained by averaging certain modified Whittaker functions over the action of modular transformations modulo translations. To make this precise, for $s\in\C$ and $y\in\RR\setminus\{0\},$ let
$$
\mathcal{M}_s(y):=|y|^{-\frac{k}{2}}M_{\sign(y)\frac{k}{2},s-\frac{1}{2}}(|y|),
$$
where $M_{\nu,\mu}$ is the usual $M$-Whittaker function. If $\tau=x+iy,$ we let
\begin{equation}\label{PoincareComponent}
\phi_s(\tau):=\mathcal{M}_s(4\pi y)e(x),
\end{equation}
where $e(x):=e^{2\pi i x}.$ Finally, if $\Gamma$ is a finite-index subgroup of $\SL_2(\Z)$ and $\rho:=M_\rho^{-1}\infty$ is a cusp of $\Gamma$, let $\Gamma_\rho$ be the stabilizer of $\rho$ in $\Gamma,$ and we let $\widehat{\Gamma}$ and $\widehat{\Gamma}_\rho$ be the uniformizations of $\Gamma$ and $\Gamma_\rho$ respectively. 
In this notation, if $s\in\C$ and $\Re(s)>1,$ then we define the Poincar\'e series
$$
\mathcal{P}_{\rho}(\tau,m,\Gamma,k,s,\varepsilon):=\frac{1}{\Gamma(2-k)}\sum\limits_{M\in\widehat{\Gamma}_\rho\backslash\widehat{\Gamma}}\phi_s\left(\left(\frac{-m+\kappa}{t}\right) M_\rho\tau\right)\Bigg|_{k,\varepsilon}M,
$$ 
where $t$ and $\kappa$ are the cusp width parameter (see Section 3.3 of \cite{RankinBook}) respectively, and  $\Gamma(\cdot)$ is the usual Gamma-function. At special values of $s,$ these Poincar\'e series are harmonic Maass forms (see Definition 4.2 of \cite{BringmannOnoPoincare}).

\begin{lemma}[Lemma 3.1 of \cite{BringmannOnoPoincare}]
	If $k<0,$ then the Poincar\'e series $\mathcal{P}_{\rho}(\tau,m,\Gamma,k,\frac{2-k}{2},\varepsilon)$ is uniformly and absolutely convergent and defines a harmonic Maass form on $\Gamma$ of weight $k$ and multiplier $\varepsilon.$
\end{lemma}

These Poincar\'e series have a Fourier expansion that are generally more complicated than those of usual holomorphic modular forms (see Lemma 4.3 of \cite{HMFBook}). More precisely, if $f(\tau)$ is a harmonic Maass form of weight $k$ and multiplier $\varepsilon$ on $\Gamma,$ then it admits a Fourier expansion of the form 
$$
f(\tau)= \sum\limits_{n\gg-\infty}a_f^{+}(n)q^n+\sum\limits_{\substack{n\ll\infty \\ n\neq 0}}a_f^{-}(n)\Gamma(1-k,-4\pi n y)q^n,
$$
where $\Gamma(s,x)$ is the incomplete Gamma function. This implies that $f(\tau)$ decomposes into two pieces, its {\it holomorphic part}
$$
f^+(\tau):=\sum\limits_{n\gg-\infty}a_f^{+}(n)q^n
$$
and its {\it non-holomorphic part}
$$
f^-(\tau):=+\sum\limits_{\substack{n\ll\infty \\ n\neq 0}}a_f^{-}(n)\Gamma(1-k,-4\pi n y)q^n.
$$
Furthermore, it is important for our purposes to note that a harmonic Maass form with vanishing non-holomorphic part is a weakly holomorphic modular form. It is in this way that we will determine a Poincar\'e series expression for $1/\eta(\tau).$ Namely, we will produce a harmonic Maass Poincar\'e series, with vanishing non-holomorphic part, which then must equal $1/\eta(\tau).$

To this end, we consider the case where $\Gamma:=\SL_2(\Z), \rho:=\infty,$ and $M_\rho:=\begin{pmatrix}
	1 & 0 \\
	0 & 1
\end{pmatrix}.$ To describe the Fourier expansion of the  Poincar\'e series, we recall the generalized Kloosterman sums (see the proof of Theorem 3.2 of \cite{BringmannOnoPoincare}). If $c>0,n\geq 0,$ and $m\in\Z,$ let
$$
K_c(\varepsilon,m,n):=\sum\limits_{S}\frac{1}{\varepsilon(S)}\exp\left(\frac{2\pi i}{c}\frac{(m+\kappa)a+(n+\kappa)d}{t}\right),
$$
where the sum runs over all matrices $S=\begin{pmatrix}
	a & b \\
	c & d
\end{pmatrix}\in\SL_2(\Z)$ for which
$$
0\leq d<ct\ \ \ \text{and}\ \ \ 0\leq a<ct.
$$

\begin{proposition}
Assuming the notation and hypotheses above, the holomorphic part $\mathcal{P}^+(\tau,m,k,\varepsilon)$ of  $\mathcal{P}_\infty(\tau,m,\SL_2(\Z),k,1-k/2,\varepsilon)$ is given by
\begin{equation}\label{HolomorphicPartPoincare}
\mathcal{P}^+(\tau,m,k,\varepsilon)= q^{\frac{-m+\kappa}{t}}+a^+(0)+\sum\limits_{n\geq 0}a^+(n)q^{\frac{n+\kappa}{t}},
\end{equation}
where  $a^+(0)=0$ if $\kappa\neq 0,$ and where for all $n>0$ we have
\begin{equation}\label{Kloostermania}
a^+(n) = -i^{2-k}\cdot(2\pi)\left|\frac{(-m+\kappa)}{n+\kappa}\right|^\frac{1-k}{2}\cdot\frac{1}{t}
\sum\limits_{c>0}\frac{K_c(\varepsilon,-m,n)}{c}\cdot I_{1-k}\left(\frac{4\pi}{ct}\sqrt{|-m+\kappa||n+\kappa|}\right).
\end{equation}
\end{proposition}

\subsection{The generating function for $p(n)$ as a Poincar\'e series}

We now describe $1/\eta(\tau)$ as a Poincar\'e series. If $f(\tau)$ is a harmonic Maass form of weight $k$ whose holomorphic part is
$$
f^+(\tau):=\sum\limits_{n\gg-\infty}a_f^{+}(n)q^n,
$$
then its {\it principal part} at $\infty$ is defined as
$$
P_{f,\infty}(\tau):=\sum\limits_{n<0}a_f^+(n)q^n.
$$
The definition at other cusps is analogous (see Definition 4.2 of \cite{HMFBook}). Our goal is to construct a suitable weight $-1/2$ harmonic Maass Poincar\'e series  that has principal part $q^{-\frac{1}{24}}$, which not only shares the principal part of $1/\eta(\tau)$, but also turns out to be this weakly holomorphic modular form.

To this end, we make use of the fact that under very general conditions, a harmonic Maass form of negative weight is determined by its principal parts. Namely, let $\Gamma$ be a finite-index subgroup of $\SL_2(\Z)$ and let $\varepsilon$ be a multiplier system for $\Gamma$ which is trivial on a congruence subgroup. Under these assumptions, we have the following lemma.

\begin{lemma}\label{PrincipalPartsDetermination}
	Let $k<0$ and $f(\tau)$ a harmonic Maass form of weight $k$ on $\Gamma$ with multiplier system $\varepsilon.$ If the principal parts of $f(\tau)$ at all cusps vanish, then $f(\tau)=0.$
\end{lemma}

\begin{proof}
	Under the aforementioned conditions on $\Gamma$ and $\varepsilon,$ we find that $f(\tau)$ can be decomposed as a sum of harmonic Maass forms on subgroups of the form $\Gamma_0(N).$ By a straightforward refinement of a proposition of Bruinier and Funke (see Lemmas 2.2 and 2.3 of \cite{BringmannOnoPoincare}), we have that $f(\tau)$ is either $0,$ or is a holomorphic modular form. However, there are no negative  weight holomorphic modular forms, and so $f(\tau)=0.$
\end{proof}

As a consequence, we obtain $1/\eta(\tau)$ as a Poincar\'e series as follows.
\begin{proposition}\label{EtaPoincare}
	We have that
	$$
	\frac{1}{\eta(\tau)} = \mathcal{P}_\infty(\tau,24,-1/2,\overline{\varepsilon}),
	$$
	where $\varepsilon(\cdot)$ is the multiplier system in Lemma~\ref{lem2}.
\end{proposition}

\begin{proof}
	By Lemma~\ref{lem2}, we have that $1/\eta(\tau)$ is a weight $-1/2$ modular form on $\SL_2(\Z)$ with multiplier system $\overline{\varepsilon}.$ Furthermore, since the cusp width is $24$ and the cusp parameter is $23,$ it follows immediately from (\ref{DedekindDefinition}) that the principal part at $\infty$ is given by $q^{-1/24}.$ On the other hand, (\ref{HolomorphicPartPoincare}) implies that the principal part at $\infty$ of $\mathcal{P}_\infty(\tau,24,-1/2,\overline{\varepsilon})$ is also $q^{-1/24}.$ Therefore, noting that $\overline{\varepsilon}$ clearly vanishes on $\Gamma(24^3),$ Lemma~\ref{PrincipalPartsDetermination} gives the desired equality.
\end{proof}

\begin{remark}
Although Proposition~\ref{EtaPoincare} appears to be new to the literature, it will not come as a surprise to experts.  Indeed,  the expression in (\ref{Kloostermania}) is equivalent to Rademacher's celebrated formula for $p(n)$ as an infinite convergent sum. In other words, the method of Poincar\'e series gives a simple and direct proof of this formula that is different from Rademacher's original ``circle method'' proof \cite{Rademacher}.
\end{remark}

\subsection{Rankin--Cohen brackets}\label{RCSection}

To prove Theorem~\ref{theorem2}, we apply the method of unfolding to the modular forms $P_\nu(\tau),$ which are Rankin--Cohen brackets of $\eta(\tau)$ and the Poincar\'e series $1/\eta(\tau).$ This decomposes $P_\nu(\tau)$ as an infinite-sum of Rankin--Cohen brackets of $\eta(\tau)$ and the functions $\phi_s(\tau)$ defined in (\ref{PoincareComponent}). Here we recall some important facts about Rankin--Cohen brackets, and we explicitly compute the constituent derivatives that arise in the resulting formulas.

To this end, let $f$ and $g$ be smooth functions defined on $\HH,$ and let $k,l\in\R$ and $\nu\in\N_0.$ The $\nu$th Rankin--Cohen bracket of $f$ and $g$ is
\begin{equation}\label{RankinCohenDefinition}
	[f,g]_\nu :=  \sum_{\substack{r,s \geq 0 \\ r + s = \nu}} (-1)^r \frac{\Gamma(k+\nu)\Gamma(l+\nu)}{s!r!\Gamma(k+\nu-s)\Gamma(l+\nu-r)}\cdot D^r(f(\tau)) \cdot D^s(g(\tau)).
\end{equation}
The next proposition describes the modularity of general Rankin--Cohen brackets. Furthermore, from here on we shall frequently drop the dependence on the modular variable $\tau$ to ease exposition and simplify formulas. 
\begin{proposition}[Proposition 2.37 of \cite{HMFBook}]\label{RankinCohenModularity}
	Let $f$ and $g$ be (not necessarily holomorphic) modular forms on a subgroup $\Gamma$ of $\SL_2(\Z)$ with multiplier systems $\varepsilon_f$ and $\varepsilon_g$ respectively. Then the following are true.

\noindent
(1)  We have that $[f,g]_\nu$ is modular of weight $k+l+2\nu$ on $\Gamma$ with multiplier system $\varepsilon_f\varepsilon_g.$
		
\noindent
(2)  If $\gamma\in\SL_2(\R),$ then under the usual modular slash operator, we have
		$$
		[f|_k\gamma,g|_l\gamma]_\nu=[f,g]_\nu|_{k+l+2\nu}\gamma.
		$$
\end{proposition}

\begin{remark}
Theorem~\ref{theorem1} offers an elementary proof of Proposition~\ref{RankinCohenModularity} (1) when $f=1/\eta$ and $g=\eta$.
\end{remark}

By applying unfolding to $[1/\eta,\eta]_\nu,$ we require Rankin--Cohen brackets of the form $[\delta(\tau),\eta(\tau)]_\nu$ where
$$
\delta(\tau):=\phi_{\frac{5}{4}}\left(\frac{-\tau}{24}\right) = e\left(\frac{-x}{24}\right)\cdot\left(\frac{\pi y}{6}\right)^{1/4}M_{\frac{1}{4},\frac{3}{4}}\left(\frac{\pi y}{6}\right),
$$
where we write $\tau:=x+iy.$ To compute $D^r(\delta(\tau)),$ we recall the definition of the rising factorial. If $a\in\C$ and $n\geq 0,$ we define
\begin{equation}
(a)^{\overline{n}}:=\frac{\Gamma(a+n)}{\Gamma(a)}.
\end{equation}

\begin{lemma}\label{BasicDerivative}
	If $n\geq 0,$ then we have
	$$
	\frac{d^n}{dt^n}\left(t^{1/4}M_{\frac{1}{4},\frac{3}{4}}(t)\right) = \frac{(-1)^n}{2^n}\sum\limits_{j=0}^n2^j\binom{n}{j} \left(\frac{-3}{2}\right)^{\overline{r}} t^{\frac{1}{4}-\frac{j}{2}} M_{\frac{1}{4}-\frac{j}{2},\frac{3}{4}-\frac{j}{2}}(t).
	$$
\end{lemma}

\begin{proof}
We require the derivative identity (see \cite[13.15.15]{NIST})
\begin{align} \label{DLMF-13.15.15}
	\frac{d^n}{dt^n}\left(e^{t/2}t^{\mu-1/2}M_{\kappa,\mu}(t)\right) = (-1)^n (-2\mu)^{\overline{n}} e^{t/2} t^{\mu-\frac{1}{2}(n+1)} M_{\kappa - \frac{1}{2}n,\mu - \frac{1}{2}n}(t).
\end{align}
Writing $t^{1/4} M_{\frac{1}{4},\frac{3}{4}}(t) = e^{-t/2}(e^{t/2}t^{1/4} M_{\frac{1}{4},\frac{3}{4}}(t))$, the general Leibniz rule grants
\begin{align*}
	\frac{d^n}{dt^n}\left(e^{-t/2}(e^{t/2}t^{1/4} M_{\frac{1}{4},\frac{3}{4}}(t))\right) &= \sum_{j=0}^{n} \binom{n}{j} \frac{d^{n-j}}{dt^{n-j}}(e^{-t/2})\frac{d^j}{dt^j}(e^{t/2}t^{1/4} M_{\frac{1}{4},\frac{3}{4}}(t)).
\end{align*}
We then apply \eqref{DLMF-13.15.15} to obtain
\begin{align*}
	\frac{d^n}{dt^n}\left(e^{-t/2}(e^{t/2}t^{1/4} M_{\frac{1}{4},\frac{3}{4}}(t))\right) &= \sum_{j=0}^{n} \binom{n}{j} \frac{(-1)^{n-j}e^{-t/2}}{2^{n-j}} \times (-1)^j \left(\frac{-3}{2}\right)^{\overline{j}} e^{t/2} t^{\frac{1}{4}-\frac{j}{2}} M_{\frac{1}{4} - \frac{j}{2},\frac{3}{4} - \frac{j}{2}}(t) \\
	&= \frac{(-1)^n}{2^n} \sum_{j=0}^{n} 2^j\binom{n}{j}  \left(\frac{-3}{2}\right)^{\overline{j}} t^{\frac{1}{4}-\frac{j}{2}} M_{\frac{1}{4} - \frac{j}{2},\frac{3}{4} - \frac{j}{2}}(t).
\end{align*}
\end{proof}

The next lemma computes $D^r(\delta).$

\begin{lemma}\label{gDerivatives}
	If $r\geq 0,$ then we have
	$$
	D^r(\delta(\tau)) = \frac{1}{24^r}\cdot\left(\frac{-3}{2}\right)^{\overline{r}}\cdot e\left(\frac{-x}{24}\right)\cdot\left(\frac{\pi y}{6}\right)^{\frac{1}{4}-\frac{r}{2}} M_{\frac{1}{4}-\frac{r}{2},\frac{3}{4}-\frac{r}{2}}\left(\frac{\pi y}{6}\right).
	$$
\end{lemma}

\begin{proof}
First, we note that $D^r=\frac{1}{(2\pi i )^r}\left(\frac{1}{2}\left(\frac{\partial}{\partial x} - i\frac{\partial}{\partial y}\right)\right)^r$ and then applying the binomial theorem, we obtain 
$$
D^r(\delta(\tau))=\frac{1}{(4\pi i)^r}\sum_{j=0}^r\binom{r}{j}\left(\frac{\partial}{\partial x}\right)^{r-j} e\left(\frac{-x}{24}\right)\cdot\left(-i\frac{\partial}{\partial y}\right)^{j}\left(\frac{\pi y}{6}\right)^{1/4}M_{\frac{1}{4},\frac{3}{4}}\left(\frac{\pi y}{6}\right).
$$
By applying Lemma \ref{BasicDerivative}, we get 
$$D^r(\delta(\tau))=\left(\frac{-1}{48}\right)^r e\left(\frac{-x}{24}\right)\sum_{j=0}^r(-1)^j\binom{r}{j}\sum_{s=0}^{j} 2^s\binom{j}{s}\left(\frac{-3}{2}\right)^{\overline{s}}\left(\frac{\pi y}{6}\right)^{\frac{1}{4}-\frac{s}{2}}M_{\frac{1}{4}-\frac{s}{2},\frac{3}{4}-\frac{s}{2}}\left(\frac{\pi y}{6}\right).
$$
After changing the order of summation, we obtain
$$D^r(\delta(\tau))=\left(\frac{-1}{48}\right)^r e\left(\frac{-x}{24}\right)\sum_{s=0}^r2^s\left(\frac{-3}{2}\right)^{\overline{s}}\left(\sum_{j=s}^{r} (-1)^j\binom{r}{j}\binom{j}{s}\right)\left(\frac{\pi y}{6}\right)^{\frac{1}{4}-\frac{s}{2}}M_{\frac{1}{4}-\frac{s}{2},\frac{3}{4}-\frac{s}{2}}\left(\frac{\pi y}{6}\right).
$$
We complete the proof of the lemma by observing that the sum $\sum_{j=s}^{r} (-1)^j\binom{r}{j}\binom{j}{s}=(-1)^r,$ when $s=r$, otherwise it is zero.
\end{proof}

\subsection{Petersson inner products and unfolding}\label{PeterssonSubsection}

Here we describe Petersson inner products and the method of unfolding which is essential to the proof of Theorem~\ref{theorem2}. For convenience, we restrict our attention to $\SL_2(\Z).$ 
Let $f$ and $g$ be two holomorphic modular forms on $\SL_2(\Z)$ with trivial multiplier, and assume $g$ is a cusp form.  We define the Petersson inner product of $f$ and $g$ as
\begin{equation}
\langle f,g\rangle = \int\int_{\HH/\SL_2(\Z)} f(x+iy)\overline{g}(x+iy)y^{k}\frac{dxdy}{y^2}.
\end{equation}
This inner product gives the vector space of cusp forms the structure of a finite-dimensional Hilbert space. Therefore, in order to describe $P_\nu(\tau)$ as a linear combination of cusp forms, we aim to determine $\langle P_\nu(\tau),f\rangle,$  where $f$ runs over an orthogonal basis of the space of cusp forms.

We now describe the method of unfolding.

\begin{lemma}\label{UnfoldingMethod}
	If $f$ is a cusp form of weight $2\nu$ on $\SL_2(\Z)$ with trivial multiplier, then
	$$
	\langle [1/\eta,\eta]_\nu, f\rangle = \int_0^\infty\int_0^1 [\delta, \eta]_\nu(x+iy)\overline{f}(x+iy)y^{2\nu-2}dxdy,
	$$
	where $\delta(\tau):=\phi_{\frac{5}{4}}\left(\frac{-\tau}{24}\right).$
\end{lemma}

\begin{proof}
By Proposition \ref{EtaPoincare}, we have that the Poincar\'e series for $1/\eta$ is
$$
\frac{1}{\eta} = \mathcal{P}_\infty(\tau,24,-1/2,\overline{\varepsilon})=\frac{1}{\Gamma(\frac{5}{2})}\sum\limits_{M\in\hat{\Gamma}_\infty\backslash\hat{\Gamma}}\phi_s\left(\left(\frac{-1}{24}\right) \tau\right)\Bigg|_{\frac{-1}{2},\overline{\varepsilon}}M.
$$
Ignoring the convergence issues, we have that
\begin{align*}[1/\eta,\eta]_\nu&=\frac{1}{\Gamma(\frac{5}{2})}\sum\limits_{M\in\hat{\Gamma}_\infty\backslash\hat{\Gamma}}\left[e\left(\frac{x}{24}\right)\left(\frac{\pi y}{6}\right)^{\frac{1}{4}}M_{\frac{1}{4},\frac{3}{4}}\left(\frac{\pi y}{6}\right)\Bigg|_{\frac{-1}{2},\overline{\varepsilon}}M,\ \ \eta(\tau)\right]_\nu\\
	&=\frac{1}{\Gamma(\frac{5}{2})}\sum\limits_{M\in\hat{\Gamma}_\infty\backslash\hat{\Gamma}}\left[\delta(\tau)\Bigg|_{\frac{-1}{2},\overline{\varepsilon}}M,\ \ \eta(\tau)\Bigg|_{\frac{1}{2},\varepsilon}M^{-1}\Bigg|_{\frac{1}{2},\varepsilon}M\right]_\nu.
\end{align*}

Since $\eta$ is modular form of weight $\frac{1}{2}$
on $\SL_2(\mathbb{Z})$ with a multiplier system, Proposition \ref{RankinCohenModularity} yields that
$$
[1/\eta,\eta]_\nu=\frac{1}{\Gamma(\frac{5}{2})}\sum\limits_{M\in\hat{\Gamma}_\infty\backslash\hat{\Gamma}}[\delta(\tau),\  \eta(\tau)]_\nu\Bigg|_{2\nu}M. 
$$
By the definition of the Petersson inner product and modularity of $f$, we have that
\begin{align*}
	\langle [1/\eta,\eta]_\nu, f\rangle &= \int\int_{\mathbb{H}/\SL_2(\Z)}\frac{1}{\Gamma(\frac{5}{2})}\sum\limits_{M\in\hat{\Gamma}_\infty\backslash\hat{\Gamma}}[\delta(\tau),\  \eta(\tau)]_\nu\Bigg|_{2\nu}M \cdot \overline{f}(x+iy)y^{2\nu}\frac{dxdy}{y^2}\\
	&=\frac{1}{\Gamma(\frac{5}{2})}\sum\limits_{M\in\hat{\Gamma}_\infty\backslash\hat{\Gamma}}\int\int_{M^{-1}\left(\mathbb{H}/\SL_2(\Z)\right)}[\delta(\tau),\  \eta(\tau)]_\nu \cdot \overline{f}(x+iy)y^{2\nu}\frac{dxdy}{y^2}.
\end{align*}
We complete the  proof by noting the fact that $\mathbb{H}/\Gamma_\infty=\bigcup_{M\in\hat{\Gamma}_\infty\backslash\hat{\Gamma}}M^{-1}\left(\mathbb{H}/\SL_2(\Z)\right)$.
\end{proof}

We now expand the integral in Lemma~\ref{UnfoldingMethod} as a finite sum of integrals.

\begin{lemma}\label{PeterssonUnfolding-Intermediate}
	If $\nu\geq 6$ and $f(\tau)\in S_{2\nu}$  has real Fourier coefficients, then we have
	\begin{align*}
	\langle [1/\eta,\eta]_\nu,f\rangle & = \frac{1}{24^\nu\Gamma(\frac{5}{2})}\sum\limits_{\substack{n\geq 1 \\ r+s=\nu \\ r,s\geq 0}} \frac{(-1)^r\Gamma(\nu-\frac{1}{2})\Gamma(\nu+\frac{1}{2})\Gamma(r-\frac{3}{2})}{s!r!\Gamma(r-\frac{1}{2})\Gamma(s+\frac{1}{2})\Gamma(-\frac{3}{2})}\cdot\leg{12}{n}\cdot n^{2s}\cdot a_f\left(\frac{n^2-1}{24}\right) \\
	&\times \int_0^\infty \left(\frac{\pi y}{6}\right)^{\frac{1}{4}-\frac{r}{2}} M_{\frac{1}{4}-\frac{r}{2},\frac{3}{4}-\frac{r}{2}}\left(\frac{\pi y}{6}\right)\cdot\exp\left(\frac{-\pi(2n^2-1)y}{12}\right)\cdot y^{2\nu-2}dy.
	\end{align*}
\end{lemma}

\begin{proof}By Lemma \ref{UnfoldingMethod} and the definition of Rankin--Cohen bracket, we have that
	$$
	\langle [1/\eta,\eta]_\nu, f\rangle =\frac{1}{\Gamma(\frac{5}{2})} \int_0^\infty\int_0^1 \sum\limits_{\substack{r+s=\nu \\ r,s\geq 0}}\frac{(-1)^r\Gamma(\nu-\frac{1}{2})\Gamma(\nu+\frac{1}{2})}{s!r!\Gamma(r-\frac{1}{2})\Gamma(s+\frac{1}{2})}D^r(\delta(\tau))\cdot D^s(\eta(\tau))\cdot\overline{f}(\tau)y^{2\nu}\frac{dxdy}{y^2}.
	$$
	Applying Lemma \ref{gDerivatives} and using Euler's Pentagonal Number Theorem in the form $ \eta(\tau)=\sum\limits_{n\geq 0}\leg{12}{n} q^{\frac{n^2}{24}},$ and expanding out the $s$-th derivatives of $\eta,$ we obtain 
	\begin{align*}
		\langle [1/\eta,\eta]_\nu, f\rangle =&\frac{1}{\Gamma(\frac{5}{2})} \int_0^\infty\int_0^1 \sum\limits_{\substack{r+s=\nu \\ r,s\geq 0}}\frac{(-1)^r\Gamma(\nu-\frac{1}{2})\Gamma(\nu+\frac{1}{2})}{s!r!\Gamma(r-\frac{1}{2})\Gamma(s+\frac{1}{2})}
		\cdot\frac{1}{24^r}\cdot\left(\frac{-3}{2}\right)^{\overline{r}}\cdot e\left(\frac{-x}{24}\right)\\
		&\times\left(\frac{\pi y}{6}\right)^{\frac{1}{4}-\frac{r}{2}} M_{\frac{1}{4}-\frac{r}{2},\frac{3}{4}-\frac{r}{2}}\left(\frac{\pi y}{6}\right)\cdot\frac{1}{24^s}\sum_{n\geq 0}\left(\frac{12}{n}\right)n^{2s}q^{\frac{n^2}{24}}\cdot\overline{f}(\tau)y^{2\nu}\frac{dxdy}{y^2}.
	\end{align*}
	We next expand $f$ as Fourier series, and a simple algebraic manipulation gives 
	\begin{align*}
		\langle [1/\eta,\eta]_\nu, f\rangle &=\frac{1}{24^\nu\Gamma(\frac{5}{2})} \sum\limits_{\substack{r+s=\nu \\ r,s\geq 0}}\frac{(-1)^r\Gamma(\nu-\frac{1}{2})\Gamma(\nu+\frac{1}{2})\Gamma(r-\frac{3}{2})}{s!r!\Gamma(r-\frac{1}{2})\Gamma(s+\frac{1}{2})\Gamma(\frac{-3}{2})}
		\cdot \left(\frac{\pi y}{6}\right)^{\frac{1}{4}-\frac{r}{2}} M_{\frac{1}{4}-\frac{r}{2},\frac{3}{4}-\frac{r}{2}}\left(\frac{\pi y}{6}\right)\\
		&\times\sum_{n\geq 1}\sum_{m\geq1}\int_0^\infty\int_0^1\left(\frac{12}{n}\right)n^{2s}a_f(m) \cdot e\left(\frac{24m+n^2-1}{24}\right)\exp\left(\frac{-\pi(24m+n^2)}{12}\right)y^{2\nu}\frac{dxdy}{y^2}.
	\end{align*}
	By observing that the integral with respect to $x$ vanishes unless $(n^2-1)-24m=0$ and noting the fact that $24| (n^2-1)$ if $\gcd(n,12)=1$, we get 
	\begin{align*}
		\langle [1/\eta,\eta]_\nu,f\rangle & = \frac{1}{24^\nu\Gamma(\frac{5}{2})}\sum\limits_{\substack{n\geq 1 \\ r+s=\nu \\ r,s\geq 0}} \frac{(-1)^r\Gamma(\nu-\frac{1}{2})\Gamma(\nu+\frac{1}{2})\Gamma(r-\frac{3}{2})}{s!r!\Gamma(r-\frac{1}{2})\Gamma(s+\frac{1}{2})\Gamma(-\frac{3}{2})}\cdot\leg{12}{n}\cdot n^{2s}\cdot a_f\left(\frac{n^2-1}{24}\right) \\
		&\times \int_0^\infty \left(\frac{\pi y}{6}\right)^{\frac{1}{4}-\frac{r}{2}} M_{\frac{1}{4}-\frac{r}{2},\frac{3}{4}-\frac{r}{2}}\left(\frac{\pi y}{6}\right)\cdot\exp\left(\frac{-\pi(2n^2-1)y}{12}\right)\cdot y^{2\nu-2}dy.
	\end{align*}
	This completes the proof of the lemma.
\end{proof}

\subsection{Twisted quadratic Dirichlet series and Petersson inner products}

We now turn to the final technical task. We describe the Petersson inner product of $[1/\eta,\eta]_\nu$ with a cusp form $f$ as an infinite weighted sum of values of twisted quadratic Dirichlet series. 
Namely, we prove the following key inner product formula.

\begin{theorem}\label{Final_lemma}
	If $f(\tau)\in S_{2\nu}$ has real Fourier coefficients, then we have
	$$
	D_f= 24^{\nu}\cdot \langle [1/\eta,\eta]_\nu,f\rangle.
	$$
\end{theorem}

To prove this theorem, we require a number of lemmas.  The following lemma describes the evaluation of the integral in Lemma~\ref{PeterssonUnfolding-Intermediate}.

\begin{lemma}\label{PeterssonUnfolding-2F1}
	If $n>1$ and $0\leq r\leq \nu,$ then
	\begin{align*}
	&\int_0^\infty \left(\frac{\pi y}{6}\right)^{\frac{1}{4}-\frac{r}{2}} M_{\frac{1}{4}-\frac{r}{2},\frac{3}{4}-\frac{r}{2}}\left(\frac{\pi y}{6}\right)\cdot\exp\left(\frac{-\pi(2n^2-1)y}{12}\right)\cdot y^{2\nu-2}dy\\
	 &\hskip1.5in= \left(\frac{6}{\pi}\right)^{2\nu-1}\cdot\frac{\Gamma(2\nu-r+\frac{1}{2})}{n^{4\nu-2r+1}}\cdot{_2F_1}\left(\begin{array}{cc}
		1 & \frac{1}{2}+2\nu-r\\
		~& \frac{5}{2}-r \\
	\end{array}\mid \frac{1}{n^2}\right),
	\end{align*}
	where ${_2F_1}$ is the classical hypergeometric function.
\end{lemma}

\begin{proof}We first perform the change variables $t = \frac{\pi y}{6}$ to obtain the integral
	\begin{align*}
		\left(\frac{6}{\pi}\right)^{2\nu - 1}\int_0^\infty t^{\frac{1}{4}-\frac{r}{2}+2\nu-2} M_{\frac{1}{4}-\frac{r}{2},\frac{3}{4}-\frac{r}{2}}(t)\cdot\exp\left(-\left(n^2 - \frac{1}{2}\right)t\right)dt.
	\end{align*}
	We then make use of the following integral identity (see \cite[13.23.1]{NIST})
	\begin{align*}
		\int_0^\infty t^{\lambda-1} M_{\kappa,\mu}(t)\cdot\exp\left(-\left(z - \frac{1}{2}\right)t\right)dt &= \frac{\Gamma(\lambda+\mu+\frac{1}{2})}{z^{\lambda+\mu+\frac{1}{2}}} \cdot {_2F_1}\left(\begin{array}{cc}
			\frac{1}{2}+\mu-\kappa & \frac{1}{2}+\lambda+\mu \\
			~& 1 + 2\mu \\
		\end{array}\mid \frac{1}{z}\right),
	\end{align*}
	which grants
	\begin{displaymath}
	\begin{split}
		&\left(\frac{6}{\pi}\right)^{2\nu - 1}\int_0^\infty t^{\frac{1}{4}-\frac{r}{2}+2\nu-2} M_{\frac{1}{4}-\frac{r}{2},\frac{3}{4}-\frac{r}{2}}(t)\cdot\exp\left(-\left(n^2 - \frac{1}{2}\right)t\right)dt \\
		&\hskip1.5in= \left(\frac{6}{\pi}\right)^{2\nu - 1} \cdot \frac{\Gamma(2\nu - r + \frac{1}{2})}{n^{4\nu - 2r+1}} \cdot{_2F_1}\left(\begin{array}{cc}
			1 & \frac{1}{2}+2\nu-r\\
			~& \frac{5}{2}-r \\
		\end{array}\mid \frac{1}{n^2}\right).
	\end{split}
	\end{displaymath}
\end{proof}
By Lemmas~\ref{PeterssonUnfolding-Intermediate} and \ref{PeterssonUnfolding-2F1}, if $f(\tau)\in S_{2\nu}$ has real Fourier coefficients, then
$$
\langle[1/\eta,\eta]_\nu,f\rangle = \frac{1}{24^\nu}\cdot\left(\frac{6}{\pi}\right)^{2\nu-1}\sum\limits_{n>1} a_f\left(\frac{n^2-1}{24}\right)\cdot\frac{\leg{12}{n}}{n^{2\nu+1}}\cdot\omega(\nu,n),
$$
where
$$
\omega(\nu,n):=\sum\limits_{r=0}^\nu \frac{(-1)^r\Gamma(\nu-\frac{1}{2})\Gamma(\nu+\frac{1}{2})\Gamma(\frac{1}{2}+2\nu-r)}{(\nu-r)!r!\Gamma(r-\frac{1}{2})\Gamma(\nu-r+\frac{1}{2})}\cdot \left(\frac{-3}{2}\right)^{\overline{r}}\cdot{_2F_1}\left(\begin{array}{cc}
	1 & \frac{1}{2}+2\nu-r\\
	~& \frac{5}{2}-r \\
\end{array}\mid \frac{1}{n^2}\right).
$$

The following two lemmas are necessary to rewrite $\omega(\nu,n)$ in a form which will give rise to the twisted Dirichlet sums $D(f,s).$ 

\begin{lemma}\label{Combinatorics1}
	If $\nu\geq 1$ and $0\leq r\leq \nu,$ then
	$$
	{_2F_1}\left(\begin{array}{cc}
		1 & \frac{1}{2}+2\nu-r\\
		~& \frac{5}{2}-r \\
	\end{array}\mid z\right) = (1-z)^{1-2\nu}\sum\limits_{i=0}^{2\nu-2}(-1)^i\binom{2\nu-2}{i}\frac{\left(\frac{3}{2}-r\right)^{\overline{i}}}{\left(\frac{5}{2}-r\right)^{\overline{i}}}z^i.
	$$
\end{lemma}

\begin{proof}
We first apply Euler's transformation formula (see \cite[2.2.5]{AAR}) to obtain
\begin{align*}
	{_2F_1}\left(\begin{array}{cc}
		1 & \frac{1}{2}+2\nu-r\\
		~& \frac{5}{2}-r \\
	\end{array}\mid z\right) = (1-z)^{1-2\nu}{_2F_1}\left(\begin{array}{cc}
		\frac{3}{2}-r & 2-2\nu\\
		~& \frac{5}{2}-r \\
	\end{array}\mid z\right).
\end{align*}
We write this as an explicit series,
\begin{align*}
	{_2F_1}\left(\begin{array}{cc}
		\frac{3}{2}-r & 2-2\nu\\
		~& \frac{5}{2}-r \\
	\end{array}\mid z\right) = \sum_{i=0}^{\infty} \frac{(\frac{3}{2} - r)^{\overline{i}}(2 - 2\nu)^{\overline{i}}}{i!(\frac{5}{2} - r)^{\overline{i}}} z^i,
\end{align*}
and observe that $2 - 2\nu \leq 0$, whereby $(2 - 2\nu)^{\overline{i}} = 0$ if $i > 2\nu - 2$. Thus, by a straightforward rewriting,
\begin{align*}
	{_2F_1}\left(\begin{array}{cc}
		\frac{3}{2}-r & 2-2\nu\\
		~& \frac{5}{2}-r \\
	\end{array}\mid z\right) = \sum_{i=0}^{2\nu - 2} (-1)^i\binom{2\nu - 2}{i} \frac{(\frac{3}{2} - r)^{\overline{i}}}{(\frac{5}{2} - r)^{\overline{i}}} z^i,
\end{align*}
and the claim follows.
\end{proof}

\begin{lemma}\label{Combinatorics2}
	If $\nu\geq 1,$ then we have
	\begin{displaymath}
	\begin{split}
		&\sum\limits_{i=0}^{2\nu-2}(-1)^i z^i\binom{2\nu-2}{i}\sum\limits_{r=0}^\nu\frac{(-1)^r\left(\frac{-3}{2}\right)^{\overline{r}}\Gamma(2\nu+\frac{1}{2}-r)}{r!(\nu-r)!\Gamma(r-\frac{1}{2})\Gamma(\nu-r+\frac{1}{2})}\cdot\frac{\left(\frac{3}{2}-r\right)^{\overline{i}}}{\left(\frac{5}{2}-r\right)^{\overline{i}}} \\
		&\hskip1.5in=\sum\limits_{i=0}^{\nu-2}(-1)^iz^i\binom{2\nu-2}{i}\frac{(\nu-i-1)^{\overline{\nu}}\left(\frac{3}{2}\right)^{\overline{i}}}{(-\frac{1}{2}-i)^{\overline{\nu}}\Gamma(\frac{-1}{2})\left(\frac{5}{2}\right)^{\overline{i}}}.
	\end{split}
	\end{displaymath}
\end{lemma}

\begin{proof}
	For fixed $0 \leq i \leq 2\nu - 2$, set
	\begin{align*}
		\alpha_r &= \frac{(-1)^r\left(\frac{-3}{2}\right)^{\overline{r}}\Gamma(2\nu+\frac{1}{2}-r)}{r!(\nu-r)!\Gamma(r-\frac{1}{2})\Gamma(\nu-r+\frac{1}{2})}\cdot\frac{\left(\frac{3}{2}-r\right)^{\overline{i}}}{\left(\frac{5}{2}-r\right)^{\overline{i}}}.
	\end{align*}
	We find by straightforward algebraic manipulation that
	\begin{align*}
		\frac{\alpha_{r+1}}{\alpha_r} = \frac{(r-\nu)(r-\frac{3}{2}-i)(r-\nu+\frac{1}{2})}{(r+1)(r-\frac{1}{2}-i)(r-2\nu+\frac{1}{2})},
	\end{align*}
	whereby we see that, since $\alpha_r = 0$ if $r \geq \nu + 1$, 
	\begin{align*}
		\sum_{r=0}^{\nu} \alpha_r = \alpha_0 \cdot{_3F_2}\left(\begin{array}{ccc}
			-\nu & - \frac{3}{2} - i & -\nu + \frac{1}{2} \\[4pt]
			~& -\frac{1}{2} - i & -2\nu + \frac{1}{2} \\
		\end{array}\mid 1\right).
	\end{align*}
	We then apply the Pfaff-Saalsch\"{u}tz identity (\cite[Thm. 2.2.6]{AAR}) to obtain
	\begin{align*}
		\sum_{r=0}^{\nu} \alpha_r = \alpha_0 \cdot \frac{(\nu - i - 1)^{\overline{\nu}}}{(-\frac{1}{2}-i)^{\overline{\nu}}(\nu + \frac{1}{2})^{\overline{\nu}}}.
	\end{align*}
	Writing out $\alpha_0$ explicitly, then noting that $(\nu - i - 1)^{\overline{\nu}} = 0$ if $i \geq \nu -1$, we obtain the lemma.
\end{proof}

With these lemmas, we prove Theorem~\ref{Final_lemma}.

\begin{proof}[Proof of Theorem~\ref{Final_lemma}]
Combining Lemmas~\ref{PeterssonUnfolding-Intermediate} and \ref{PeterssonUnfolding-2F1}, we get
\begin{align*}
	\langle[1/\eta,\eta]_\nu,f\rangle =& \frac{1}{24^\nu\Gamma\left(\frac{5}{2}\right)}\cdot\left(\frac{6}{\pi}\right)^{2\nu-1}\sum\limits_{n>1} a_f\left(\frac{n^2-1}{24}\right)\cdot\frac{\leg{12}{n}}{n^{2\nu+1}}\cdot\sum\limits_{r=0}^\nu \frac{(-1)^r\Gamma(\nu-\frac{1}{2})\Gamma(\nu+\frac{1}{2})}{(\nu-r)!r!\Gamma(r-\frac{1}{2})}\\
	&\times \frac{\Gamma(\frac{1}{2}+2\nu-r)}{\Gamma(\nu-r+\frac{1}{2})}\cdot \left(\frac{-3}{2}\right)^{\overline{r}}\cdot{_2F_1}\left(\begin{array}{cc}
		1 & \frac{1}{2}+2\nu-r\\
		~& \frac{5}{2}-r \\
	\end{array}\mid \frac{1}{n^2}\right).
\end{align*}
Applying Lemmas~\ref{Combinatorics1} and \ref{Combinatorics2} consecutively with $z=1/n^2$, we obtain
\begin{align*}
	\langle[1/\eta,\eta]_\nu,f\rangle =& \frac{\Gamma(\nu-\frac{1}{2})\Gamma(\nu+\frac{1}{2})}{24^\nu\Gamma\left(\frac{5}{2}\right)}\cdot\left(\frac{6}{\pi}\right)^{2\nu-1}\sum\limits_{n>1} a_f\left(\frac{n^2-1}{24}\right)\cdot\frac{\leg{12}{n}}{n^{2\nu+1}}\cdot\frac{1}{(1-\frac{1}{n^2})^{2\nu-1}}\\
	&\times \sum_{i=0}^{\nu-2}(-1)^i\binom{2\nu-2}{i}\frac{1}{n^{2i}}
	\cdot\frac{(\nu-i-1)^{\overline{\nu}}\left(\frac{3}{2}\right)^{\overline{i}}}{(-\frac{1}{2}-i)^{\overline{\nu}}\cdot\Gamma(-\frac{1}{2})\cdot(\frac{5}{2})^{\overline{i}}}.
\end{align*}
Using the Taylor series expansion of $\left(\frac{1}{1-1/n^2}\right)^{2\nu-1}$ at infinity ($z=\frac{1}{n^2}$ at $z=0$) and changing the order of summation and performing straightforward manipulations, we obtain
\begin{align*}
	\langle[1/\eta,\eta]_\nu,f\rangle =& \frac{\Gamma(\nu-\frac{1}{2})\Gamma(\nu+\frac{1}{2})}{2\sqrt{\pi}\cdot 24^\nu\Gamma\left(\frac{5}{2}\right)}\cdot\left(\frac{6}{\pi}\right)^{2\nu-1} \sum_{i=0}^{\nu-2}\sum_{m=0}^\infty (-1)^{i+1}\binom{2\nu-2}{i}\binom{2\nu+m-2}{m}\frac{1}{n^{2i+2m}}
	\\
	&\times\frac{(\nu-i-1)^{\overline{\nu}}\left(\frac{3}{2}\right)^{\overline{i}}}{(-\frac{1}{2}-i)^{\overline{\nu}}\cdot(\frac{5}{2})^{\overline{i}}}\sum\limits_{n>1} a_f\left(\frac{n^2-1}{24}\right)\cdot\frac{\leg{12}{n}}{n^{2\nu+1}}\\
	&=\frac{1}{24^\nu}\sum_{i=0}^{\nu-2}\sum_{m=0}^\infty\beta(\nu,i,m)\sum\limits_{n>1} a_f\left(\frac{n^2-1}{24}\right)\cdot\frac{\leg{12}{n}}{n^{2\nu+2i+2m+1}},
\end{align*}
where 
\begin{align*}
	\beta(\nu,i,m)=\frac{(-1)^{i+1}\Gamma(\nu-\frac{1}{2})\Gamma(\nu+\frac{1}{2})}{2\sqrt{\pi}\cdot \Gamma\left(\frac{5}{2}\right)}\cdot\left(\frac{6}{\pi}\right)^{2\nu-1}\frac{(2\nu+m-2)!}{i!\cdot m!\cdot (2\nu-i-2)!}\cdot\frac{(\nu-i-1)^{\overline{\nu}}\left(\frac{3}{2}\right)^{\overline{i}}}{(-\frac{1}{2}-i)^{\overline{\nu}}\cdot(\frac{5}{2})^{\overline{i}}}.
\end{align*}
Thanks to (\ref{InfiniteSum}), this completes the proof.
\end{proof}

\section{Proof of Theorem~\ref{theorem2}}
Here we prove Theorem~\ref{theorem2} using the results of the previous section.
Thanks to Proposition~\ref{RankinCohenModularity}, we have that $P_{\nu}:=[1/\eta,\eta]_\nu$ is a holomorphic modular form of weight $2\nu$ on $\SL_2(\mathbb{Z})$. Hence, we have 
\begin{align}\label{eq4.1}
P_{\nu}=[1/\eta,\eta]_\nu=\alpha_{\nu}  E_{2\nu}(\tau)+\sum_{i}\gamma_i\ f_i(\tau),
\end{align}
where the $f_i$'s are an orthogonal basis of Hecke eigenforms for $S_{2\nu}$. It is well-known that these Hecke eigenforms   have real Fourier coefficients (see p. 105 of  \cite{Serre}), and also satisfy
$$
f_i \ | \ T_n = a_{f_i}(n)\ f_i,
$$
where $f_i(\tau)=q+\sum_{n\geq 2}a_{f_i}(n)\ q^n$. Using the orthogonality of $f_i$'s, we obtain
\begin{align}\label{eq4.2}
\langle[1/\eta,\eta]_\nu,f_i\rangle=\gamma_i\ \langle f_i,f_i\rangle.
\end{align}
By combining \eqref{eq4.1}, \eqref{eq4.2} with Theorem~\ref{Final_lemma}, we obtain an explicit description of $P_{\nu}$ involving the coefficients of $E_{2\nu}$, Hecke eigenvalues, and infinite sums of values of twisted Dirichlet series. 
Finally, by comparing the coefficients of $q^n$ and employing (\ref{series_side}) to solve for $p(n)$, we obtain the claimed recurrence relations.

\end{document}